\documentclass[twoside,a4paper,reqno]{amsart}
\def\@logofont{\footnotesize}
\def\@setaddresses{\par
  \nobreak \begingroup
  \footnotesize
  \def\author##1{\nobreak\addvspace\bigskipamount}%
  \def\\{\par\nobreak}%
  \interlinepenalty\@M
  \def\address##1##2{\begingroup
    \par\addvspace\bigskipamount\indent
    \@ifnotempty{##1}{(\ignorespaces##1\unskip) }%
    {\scshape\ignorespaces##2}\par\endgroup}%
  \def\curraddr##1##2{\begingroup
    \@ifnotempty{##2}{\nobreak\indent\curraddrname
      \@ifnotempty{##1}{, \ignorespaces##1\unskip}\/:\space
      ##2\par}\endgroup}%
  \def\email##1##2{\begingroup
    \@ifnotempty{##2}{\nobreak\indent\emailaddrname
      \@ifnotempty{##1}{, \ignorespaces##1\unskip}\/:\space
      \ttfamily##2\par}\endgroup}%
  \def\urladdr##1##2{\begingroup
    \def~{\char`\~}%
    \@ifnotempty{##2}{\nobreak\indent\urladdrname
      \@ifnotempty{##1}{, \ignorespaces##1\unskip}\/:\space
      \ttfamily##2\par}\endgroup}%
  \addresses
  \endgroup
}
\renewcommand*\subjclass[2][2010]{%
  \def\@subjclass{#2}%
  \@ifundefined{subjclassname@#1}{%
    \ClassWarning{\@classname}{Unknown edition (#1) of Mathematics
      Subject Classification; using '2000'.}%
  }{%
    \@xp\let\@xp\subjclassname\csname subjclassname@#1\endcsname
  }%
}

\usepackage{graphicx,amsmath,amssymb}
\usepackage{algorithm}
\usepackage{mathdots, comment}
\usepackage[noend]{algpseudocode}
\makeatletter
\def\BState{\State\hskip-\ALG@thistlm}
\makeatother
\allowdisplaybreaks

\theoremstyle{plain}
\newtheorem{theorem}{Theorem}[section]
\newtheorem{proposition}[theorem]{Proposition}
\newtheorem{lemma}[theorem]{Lemma}
\newtheorem{corollary}[theorem]{Corollary}
\newtheorem{conjecture}[theorem]{Conjecture}

\theoremstyle{definition}
\newtheorem{definition}[theorem]{Definition}
\newtheorem{example}[theorem]{Example}

\newtheorem{remark}[theorem]{Remark}

\begin{document}
\title{Conditions for matchability in groups and field extensions II}

\author[M. Aliabadi]{Mohsen Aliabadi}

\thanks{Department of Mathematics, University of California, San Diego, 9500 Gilman Dr, La Jolla, Ca 92093, USA, \texttt{maliabadisr@ucsd.edu}, ORCID: 0000-0001-5331-2540 .}

\begin{abstract}
We present sufficient conditions for the existence of matchings in abelian groups and their linear counterparts. These conditions lead to extensions of existing results in matching theory. Additionally, we classify subsets within abelian groups that cannot be matched. We introduce the concept of Chowla subspaces and formulate and conjecture a linear analogue of a result originally attributed to Y. O. Hamidoune \cite{19} concerning Chowla subsets. If proven true, this result would extend matchings in primitive subspaces. Throughout the paper, we emphasize the analogy between matchings in abelian groups and field extensions. We also pose numerous open questions for future research. Our approach relies on classical theorems in group theory, additive number theory and linear algebra. As the title of the paper suggests, this work is the second sequel to a previous paper \cite{5} with a similar theme. The paper is self-contained and can be read independently.

\end{abstract}
\maketitle
 \section{Introduction}
 Let $A$ and $B$ be finite subsets of the multiplicative abelian group $G$ with the same cardinality. A \textit{matching} from $A$ to $B$ is a bijection $f:A\to B$ such that $af(a)\notin A$, for every $a\in A$. 
 In this case, we say $A$ is \textit{matched} (or \textit{matchable}) to $B$. 
 The notion of matchings in abelian groups was initially introduced in \cite{16} in order to generalize a geometric property of lattices in Euclidean space. 
 Their goal was to demonstrate that a suitably selected fixed set of monomials, if sufficiently small, could be removed from a generic homogeneous polynomial through a suitable linear change of coordinates. The inquiry into sets of removable monomials from generic homogeneous polynomials via linear transformations has historical origins dating back to \cite{28}.
Matchings have been extensively studied in the context of abelian groups \cite{6,24} as well as in arbitrary group settings \cite{14}. The linear version of matchings is introduced in \cite{15} and further investigated in \cite{2,4}. A matroidal analogue of matchings is introduced and explored in \cite{Shira}. For enumerative aspects of matchings, see \cite{19}. There are still many fascinating open problems in this area. This paper continues to study some problems motivated in \cite{2,5}.
\\
 Evidently, it is necessary for the existence of a matching from $A$ to $B$ that $|A|=|B|$ and that $e\notin B$ (here $e$ denotes the neutral element of $G$). One says that a group $G$ possesses the \textit{matching property} if these necessary conditions are sufficient as well. 
 It is shown in \cite{24} that an abelian group satisfies the matching property if and only if it is either torsion-free or cyclic of prime order. In \cite{2}, it is observed that the existence of a nontrivial proper finite subgroup serves as an obstruction to the matching property in general abelian groups. In particular, the following result is established in \cite{2}.

\begin{theorem}\label{tf}
Let $G$ be an abelian group and $A$ be a finite subset of $G$ which does not contain any coset of any proper nontrivial finite subgroup of $G$. Then for every finite subset $B$ of $G$ with $|A|=|B|$ and $e\notin B$, there is a matching from $A$ to $B$.
\end{theorem}

Theorem \ref{tf} was primarily established using methods from additive number theory and combinatorics. Specifically, it relied on Kneser's addition theorem (Theorem \ref{tK}) in conjunction with Hall's Marriage Theorem. These theorems served as the key components in the classification of the matching property. Initially, they were employed in \cite{24} and later adapted in \cite{2} to prove Theorem \ref{tf}.

In this paper, we first present additional sufficient conditions for the existence of
matchings in abelian groups, wherein structures from additive number theory, such
as Kemperman’s structure theorem and Sidon subsets play prominent roles. 
In continuation of Theorem 1.1 in \cite{5}, where umatchable subsets in the group setting are addressed our main
results concerning catchable subsets appear as Theorems \ref{t1} and \ref{t2}.
All the necessary definitions and concepts from additive number theory are provided
in Section 2. Note that in Theorem \ref{t1}, $p(G)$ stands for the smallest cardinality of a
nontrivial subgroup of $G$. If $G$ is finite, then $p(G)$ equals the smallest prime divisor
of $|G|$. On the other hand, $p(G) =\infty$ if and only if $G$ is torsion free.

\begin{theorem}\label{t1}
Let $G$ be any abelian group, $A$ and $B$ be finite subsets of $G$ with the same cardinality $n$ and $e\notin B$. Assume further that one of the following conditions holds:
\begin{enumerate}
\item
$B\cap H=\emptyset$, for every proper finite subgroup $H$ of $G$.
\item
$A$ and $B$ are not contained in a union of $l$ cosets of any nontrivial proper finite subgroup of $G$ and $l\geq n-1$.
\item
$A$ and $B$ are not contained in a union of $l$ cosets of any nontrivial proper finite subgroup of $G$, $l\geq n-2$, $|\{a^n|a\in A\}|>2$ and $B$ has at least two elements whose orders are greater than $n$. 
\item
$p(G)=n$, and $A$ is not a coset of any nontrivial subgroup of $G$.
\item
$A$ is a Sidon subset.
\item
The order of every element of $B$ is greater than or equal to $n$ and $A$ is not a progression.
\item
For every nontrivial finite subgroup $H$ of $G$ and elements  $a,b\in G$, $|aH\cap A|+|bH\cap B|<|H|+1$.
\end{enumerate}
Then $A$ is matched to $B$.
\end{theorem}

In the following theorem, we establish that if $A$ is not matched to $B$, then certain additive number-theoretical structures must exist within $A$.

\begin{theorem}\label{t2}
Let $A$ be a nonempty finite subset of an abelian group $G$. Assume that there exists a subset $B$ of $G$ with $|B|=|A|$ and  $e\notin B$ so that $A$ is not matched to $B$. Then one of the following conditions holds:
\begin{itemize}
\item[i)]
$A$ contains a   progression;
\item[ii)]
$A$ contains a quasi-periodic subset. 
\end{itemize}
\end{theorem}
We now turn our attention to linear matchings. 
Given a field extension $K\subset L$, an analogous notion of matchings between two $K$-subspaces of $L$ is introduced and developed in \cite{15}.
\begin{definition}
Let $A$ and $B$  be two $n$-dimensional $K$-subspaces of $L$. An ordered basis $\mathcal{A}=\{a_1,\ldots,a_n\}$ of $A$ is said to be \textit{matched} to ordered basis $\mathcal{B}=\{b_1,\ldots,b_n\}$ of $B$ if 
\begin{align}\label{eqn100}
a_i^{-1} A\cap B\subset\langle b_1,\ldots,b_{i-1},b_{i+1},\ldots,b_n\rangle,
\end{align}
for each $1\leq i\leq n$, where $\langle b_1,\ldots,b_{i-1},b_{i+1},\ldots,b_n\rangle$ stands for the subspace of $B$ spanned $\{b_1,\ldots,b_{i-1},b_{i+1},\ldots,b_n\}$. We say $A$ is {\it{matched}} to $B$ (or $A$ is \textit{matchable} to $B$) if every ordered basis $\mathcal{A}$ of $A$ can be matched to an ordered basis $\mathcal{B}$ of $B$.
\end{definition}
Note that if \eqref{eqn100} is the case, then no $a_ib_i$ lie in $A$ and thus $a_i\mapsto b_i$ defined a matching $\mathcal{A}\to\mathcal{B}$ in the multiplicative group $L^*$ in the group setting. It is established in \cite{15} that $\mathcal{A}=\{a_1,\ldots,a_n\}$ is matched to a basis of $B$ if and only if the following dimension criteria hold for every $J\subset\{1,\ldots,n\}$:
\begin{align*}
\dim_K\bigcap_{i\in J}\left(a_i^{-1}A\cap B\right)\leq n-|J|.
\end{align*}
In particular, setting $J=\{1,\ldots,n\}$, the subspace $\overset{k}{\underset{i=1}{\bigcap}}\left(a_i^{-1}A\cap B\right)$ must be trivial which cannot occur  if $1\in B$.  Having said this, the linear analogue of the matching property is defined as follows:
\begin{definition}
A field extension $K\subset L$ is said to  have the \textit{linear  matching property} if for every finite-dimensional $K$-subspaces $A$ and $B$ of $L$ with $\dim_KA=\dim_KB$ and $ 1\notin B$, $A$ is matched to $B$.
\end{definition}
Similar to the group setting it is proved in \cite{15} that $K\subset L$ possesses the linear matching property if and only if there are no nontrivial finite intermediate extension $K\subset M\subset L$. 
So the existence of nontrivial finite intermediate extensions is an obstruction for the linear matching property. This observation is formally stated and proven in  \cite{2} as follows:
\begin{theorem}\label{t1.6}
Let $K\subset L$ be a field extension and $A$ be an $n$-dimensional subspace of $L$ which does not contain any linear translate of a nontrivial finite dimensional subfield of $L$. Then $A$ is matched to any $n$-dimensional $K$-subspace of $L$ provided $1\notin B$. 
\end{theorem}
\begin{remark}
Notice that if $K\subset L$ is a field extension, $M$ is an intermediate subfield, and $l\in L^*$, a linear translate of $M$ is a $K$-subspace of the form $lM$.
\end{remark}

The proof of Theorem \ref{t1.6} is essentially based on the adaptation of methods employed in \cite{15} from the theory of matroids, specifically Rado's Theorem \cite{27}, and a linear version of Kneser's theorem (Theorem \ref{tK}). 

In continuation of Theorem 1.4 in \cite{5}, where in the linear settings, a sufficient condition for unmatchable subspaces is addressed, we come up with sufficient conditions for the existence of matchable subspaces, similar to Theorem \ref{t1} in the group setting. Likewise, certain linear structures of subspaces, such as Sidon subspaces, are leveraged in our results. 
Note that all necessary definitions are presented in Section 2. In Theorem \ref{t1.8}, we assume that $p(K,L)$ denotes the smallest degree of an intermediate field extension $K\subsetneq M\subset L$. Thus $p(K,L)\geq2$, and $p(K,L)=\infty$ if the extension is purely transcendental. All other necessary definitions concerning Theorem \ref{t1.8} and \ref{t1.9} are provided in Section 2. It is worth mentioning that Theorem \ref{t1.8} parts (1), (2), (3) may be seen as a linear analogue of Theorem \ref{t1} parts (1), (4) and (5), respectively.

\begin{theorem}\label{t1.8}
Let $K\subset L$ be a field extension. Let $A$ and $B$ be two $n$-dimensional $K$-subspaces in  $L$ with $1\notin B$ and $n>1$. Assume further that one of the following conditions holds:
\begin{enumerate}
\item
$B\cap M=\{0\}$, for every finite dimensional intermediate field $K\subset M\subsetneq L$, and either  $n\leq 5$ or every algebraic element of $L$ is separable over $K$. 
\item
$n=p(K,L)$ and $A$ is not a linear translate of any nontrivial intermediate field $K\subset M\subsetneq L$. 
\item
$A$ is a Sidon subspace. 
\end{enumerate}
Then $A$ is matched to $B$.
\end{theorem}

The following is one of our primary results, wherein unmatchable subspaces are characterized by utilizing the linear structure of subspace $B$. Note that it may be seen as a linear version of Theorem \ref{t2}.

\begin{theorem}\label{t1.9}
Let $K\subset L$ be a field extension. Let $A$ be an $n$-dimensional $K$-subspace of $L$. Assume that there exists an $n$-dimensional $K$-subspace $B$ with $1\notin B$ so that $A$ is not matched to $B$. Then $B\oplus K$ contains a subspace whose 1-atom is a nontrivial intermediate field $K\subsetneq M\subset L$. 
\end{theorem}

\textbf{Organization of the paper:} Section 2 encompasses all the necessary concepts and results drawn from additive number theory, matching theory, and their linear counterparts. These notions and results serve as the foundation for the proofs presented in our work. Section 3 is dedicated to presenting the proofs of our main results, specifically those related to matchings within the context of abelian groups. Section 4 is focused on providing the proofs related to linear matchings. Finally, in Section 5, we present two conjectures concerning linear matchings. We first explore the possibility of a linear counterpart to a result we have previously proven in the context of matchings within groups (Theorem \ref{t1} part (7)). Next, we introduce the concept of Chowla subspaces as a linear analogue of Chowla subsets, and we conjecture whether a similar result to matchings in Chowla subsets can be extended to matchings in Chowla subspaces.

\section{Preliminaries}
\subsection{Abelian additive theory}
We commence this section by providing definitions and presenting a few results from additive number theory, along with cardinal criteria for matchable sets. Given an  abelian group $G$ and nonempty subsets $A$, $B$ of $G$, we define their product as
\[AB=\{ab\mid a\in A,\ b\in B\}.\]
The subset $A\subset G$ is called a \textit{Sidon subset} if all pairwise products of its elements are distinct. That is, the equation $ab=cd$ has only the trivial solution $\{a,b\}=\{c,d\}$ in $A$. A nonempty subset $A$ of $G$, not containing 1, is called a {\it{Chowla subset}} if
the order of every element of $A$ is $\geq$ $|A|+1$. A \textit{progression} of length $k$, intitial term $a$ with common ratio $d$ is a subset of $G$ of the form $\{a,ad,ad^2,\ldots,ad^{k-1}\}$.  Let $H$ be a nontrivial subgroup of $G$. The subset $A\subset G$ is called \textit{$H$-periodic} if it can be expressed as the union of some $H$-cosets. If $A$ is $H$-periodic, every element $a\in AH\setminus {A}$ is called an {\it{$H$-hole}} and thus the number of $H$-holes in $A$ in view of Proposition 4.1(v) in \cite{17} is 

$$
|AH|-|A|=|AH\setminus {A}|.
$$

A \textit{quasi-periodic decomposition} of $A$ with \textit{quasi-period} $H$ is a partition $A=A_1\cup A_0$ of $A$ into two disjoint  (two possibly empty) subsets such that $A_1$ is $H$-periodic or an $H$-coset. A set $A\subset G$ is \textit{quasi-periodic} if $A$ has a quasi-periodic decomposition $A=A_1\cup A_0$ with $A_1$ nonempty.

The objective of this subsection is to introduce and elucidate one of the most fundamental theorems in Inverse Additive Theory: Kneser's Theorem, along with its numerous variants, extensions, and consequences. There exist several equivalent ways to present Kneser's Theorem. We begin with the following form of this theorem from \cite[p.115]{25}. See \cite{17} for further discussion about Kneser's Theorem.

\begin{theorem}[Kneser's Theorem]\label{tK}
Let $G$ be an abelian group and $A$ and $B$ be nonempty subsets of $G$.  Then
\[|AB|\geq|A|+|B|-|H|,\]
where $H$ is the stabilizer of $AB$, i.e. 
\[H=\{x\in G\mid xAB\subset AB\}.\]
\end{theorem}
Theorem \ref{tK} has the following equivalent formulation whose proof shall be found in \cite{25}.
\begin{theorem}
Let $G, A$ and $B$ be as in Theorem \ref{tK}. If $|AB|<|A|+|B|-1$, then 
$$
|AB|\geq|AH|+|BH|-|H|,
$$
where $HAB=AB$.
\end{theorem}

The following result will be used in the proof of Theorem \ref{t1} part (7). It is a consequence of  Kneser's Theorem together with the {\it{Pigeonhole bounds}} for product sets.
\begin{proposition}\label{p2.3}
Let  $G$ be an abelian group and $A$ and $B$  be nonempty finite subsets of $G$. Assume that for every nontrivial finite subgroup $H$ of $G$ and all elements $a,b\in G$, we have 
\[|aH\cap A|+|bH\cap B|\leq|H|+1.\]
Then for all nonempty subsets $S\subset A$ and $T\subset B$, we have $|ST|\geq|S|+|T|-1$.
\end{proposition}
\begin{proof}
Assume to the contrary, there exist subsets $S\subset A$, $T\subset B$ such that 
\begin{align}\label{eqqq*}
|ST|<|S|+|T|-1.
\end{align}
Applying Kneser's Theorem to $S$ and $T$, we have 
\begin{align}\label{eqqq**}
|ST|\geq |SH|+|TH|-|H|,
\end{align}
where $HST=ST$.

Combining \eqref{eqqq*} and \eqref{eqqq**} together with the fact that $|SH\setminus  S|=|SH|-|S|$ and $TH\setminus T|=|TH|-|T|$ we arrive at the following:
\begin{align}\label{eqqq***}
|SH\setminus S|+|TH\setminus T|\leq |H|-2.
\end{align}
It follows from \eqref{eqqq***} that there are at most $|H|-2$ $H$-holes, namely, elements from $HA\setminus A$ or $HB\setminus B$. Take $a\in S$ and $b\in T$. Since $S\subset A$ and $T\subset B$, so
\begin{align*}
|aH\cap A|+|bH\cap B|\geq &|aH\cap S|+|bH\cap T|\\
\geq& 2|H|-\text{(total number of $H$-holes)}\\
\geq&2|H|-(|H|-2)=|H|+2,
\end{align*}
contradicting our assumptions on $A$ and $B$. 
\end{proof}

The following theorem \cite{9} is a stronger form of Kneser's addition theorem in the sense that in Kneser's Theorem, the subgroup that stabilizes $AB$ appears to depend on both sets $A$ and $B$, but in the following theorem, the stabilizer actually may be seen to depend only on $A$.

\begin{theorem}\label{ta}
Let $G$ be an abelian group, and let $A\subset G$ be a finite subset of $G$. Then 
\begin{itemize}
\item
either for every finite  subset $B$ of $G$,  we have 
\begin{align*}
|AB|\geq |A|+|B|-1,
\end{align*}
\item
or there exists a subgroup $H\neq\{e\}$ of $G$ such that, for every finite subset $B$ of $G$ satisfying
\begin{align*}
|AB|<|A|+|B|-1,
\end{align*}
we have $HAB=AB$.
\end{itemize}
\end{theorem}
The following result from \cite{12} characterizes relatively large product sets. Essentially, it serves as a generalization of the Cauchy-Davenport Theorem \cite{10,13} and Chowla's Theorem \cite{11} in the context of abelian groups.
\begin{theorem}\label{t*}
Let $G$ be an abelian group and $A$ be a finite subset of $G$ containing the neutral element $e$. Suppose that for every proper subgroup $H$ of $G$ we have $A\cap H=\{e\}$. Then for any finite subset $S$ of $G$, we have
\begin{align*}
|AS|\geq \min\{|G|,|A|+|S|-1\}.
\end{align*}
\end{theorem}
Another result from \cite{12} generalizes Olson's result \cite{26} for the case of abelian groups is a strong building block employed in Theorem \ref{t1} parts (2), (3)  to ensure matchability in sets that are not contained in a union of certain number of cosets of any proper subgroup of the given abelian group.
\begin{theorem}\label{t**}
Let $m\geq1$ and $A_1,\ldots,A_m$ be finite, nonempty subsets of an abelian group $G$, such that no $A_i$ is contained in a union of $l$ cosets of any proper finite subgroup of $G$, then 
\begin{align*}
\left|A_1\cdots A_m\right|\geq \min\left\{|G|,\left(\frac{l}{l+1}+\frac{1}{(l+1)m}\right)\sum_{i=1}^m |A_i|\right\}.
\end{align*}
\end{theorem}

The following theorem is well-known as the {\it{Multiplicity Bound}} and it arises as a consequence of Kneser's Theorem. This theorem was originally proven by Kemperman \cite{23}. It provides valuable insights into the relationship between the size of the product set $AB$ to the multiplicities of elements within $AB$.
\begin{theorem}\label{tb}
Let $G$ be an abelian group with $A,B\subset G$ finite and nonempty. If $|AB|\leq |A|+|B|-k$, then $r_{A,B}(x)\geq k$, for all $x\in AB$, where 
\begin{align*}
r_{A,B}(x)=\left|\left\{(a,b)\in A\times B\mid  ab=x\right\}\right|.
\end{align*}
\end{theorem}
\begin{corollary}\label{cg}
Let $G$ be an abelian group and $C$ be a finite Sidon subset of $G$. If $A, B$ be subsets of $G$ with $AB\subset C$, then $|AB|\geq|A|+|B|-1$.
\end{corollary}
\begin{proof}
Since $C$ is a Sidon subset, clearly $r_{A,B}(x)=1$, for every $x\in AB$. Thus by the contrapositive of Theorem \ref{tb}, we deduce that $|AB|>|A|+|B|-2$, implying $|AB|\geq|A|+|B|-1$.
\end{proof}
We will take advantage of the following result of Hamidoune \cite{18} to prove Theorem \ref{t1}, part (6).
\begin{theorem}\label{tx}
Let $G$ be an arbitrary group, and let $A$ and $B$ be two proper finite subsets of $G$ such that $A\cup AB\neq A\langle B\rangle$, where $\langle B\rangle$ is the subgroup generated by $B$.  Then $|A\cup AB|\geq|A|+\min\{|B|,v(B)\}$, where $v(B)$ is the minimum order of an element in $B$.
\end{theorem}

The central role in proving Theorem \ref{t2} and classifying unmatchable sets based on their additive structures is attributed to the following profound result by Kemperman \cite{22}.

\begin{theorem}\label{ty}
Let $A$ and $B$ be finite subsets of an abelian group $G$ satisfying $|AB|\leq|A|+|B|-1$ and $\min\{|A|,|B|\}>1$. Then either $AB$ is a progression or $AB$ is quasi-periodic.
\end{theorem}

While classifying matchable subsets, it is evident that if $AB\cap A=\emptyset$, then A is matched to B. In particular, every bijection from A to B is a matching. Clearly, this represents an extreme situation. On the other end of the spectrum, we encounter a situation where $AB = A$. Then under this condition, $A$ cannot be matched to $B$. This scenario is explored in \cite{5}, and the following proposition is established. This result will be employed in the proof of Theorem \ref{t1} part (2).

\begin{proposition}\label{p***}
Let $A$ and $B$ be nonempty finite subsets of an arbitrary group $G$. Assume that $|A|\leq|B|$ and that $AB=A$. Then $B$ is a subgroup of $G$, and $A$ is a coset of $B$.
\end{proposition}
Similar to the dimension criteria for matchable bases of vector subspaces, cardinality criteria for matchable subsets are discussed and proved in \cite{4} as follows. This result is mainly constructed upon  Hall's Marriage Theorem, and Kneser's addition Theorem.
\begin{lemma}\label{l*}
Let $G$ be an abelian group and $A$ and $B$ be two finite subsets of $G$ with the same cardinality. Let $A=\{a_1,\ldots,a_n\}$. Then $A$ is matched to $B$ if and only if 
\begin{align*}
\left|\bigcap_{i\in J}\left(a_i^{-1}A\cap B\right)\right|\leq n-|J|,
\end{align*}
for all nonempty subset $J\subset\{1,\ldots,n\}$.
\end{lemma}
In addition to works in \cite{24, 14, 2} on classifying matchable sets, Hamidoune in \cite{19} proves the matchability of large classes of sets, in which the concept of Chowla subsets arises.

\begin{theorem}\label{Chowla}
Let $A, B$ be nonempty finite subsets of a group $G$ with $|A|=|B|$. If $B$ is a Chowla subset, then $A$ is matched to $B$.
\end{theorem}

\subsection{Linear versions of some addition theorems}

Turning our focus to linear adaptations of definitions and results from additive theorems, it is notable that several additive theorems that bound $|AB|$ have been transposed to a linear framework in the following sense. Let $K\subset L$ be a field extension and $A,B$ be $K$-subspaces of $L$. We denote by $\langle AB\rangle$ the $K$-subspace spanned by the product set $AB$ in $L$. One remarkable development in this line of research is the formulation of a linear counterpart to Kneser's addition theorem, as presented in \cite{21}. In this section, we will primarily delve into these bounds and their implications on the structures of $A$ and $B$.

We begin with defining the required notions used in Theorems \ref{t1.8} and \ref{t1.9} starting with the notion of Sidon subspaces.
Following \cite{7}, we say a $K$-subspace $A$ of $L$ is a {\it{Sidon subspace}} if:
\begin{align}\label{eqqq1}
\forall x\in L\setminus K,\ \dim_{K}(A\cap xA)\leq 1.
\end{align}
Note that it is discussed in \cite{7} that \eqref{eqqq1} implies 
\begin{align}\label{eqqq2}
\forall x,y,z,t\in A\setminus\{0\}, \ xy=zt\Rightarrow \{Kx,Ky\}=\{Kz,Kt\}.
\end{align}
In view of \eqref{eqqq2}, the analogy with the Sidon subsets is highlighted. 

We continue borrowing some terminologies from \cite{7} to prepare ourselves to define 1-atoms. Let $A$ be a $K$-subspace of $L$. For every nonzero finite-dimensional $K$-subspace $X$ of $L$, we denote by 
\begin{align*}
\partial_A(X)=\dim_K\langle XA\rangle-\dim_{K} X.
\end{align*}
Let $\psi$ be the set of nonzero finite-dimensional $K$-subspaces $X$ of $L$ such that $\langle XA\rangle\neq L$.  If $\psi$ is nonempty, we define the \textit{1st-connectivity} of $A$ by 
\begin{align*}
\kappa(A)=\min_{X\in\psi}\partial_A(X).
\end{align*}
If $\psi$ is empty, we set $\kappa(A)\neq-\infty$. We define  a \textit{1st-fragment} of $A$ to be an $M\in\psi$ with $\partial_A(M)=\kappa(A)$. A 1st-fragment with minimum dimension is called a {\it{1-atom}}. 
\\

We now state a linear analogue of Kneser's addition Theorem and its variants and consequences. We begin with the following result from \cite{21} in which the separability condition is imposed on our field.

\begin{theorem}\label{tt'}
Let $K\subset L$ be a field extension in which every algebraic element of $L$ is separable over $K$. Let $A,B\subset L$ be nonzero finite dimensional $K$-subspaces of $L$ and $M$ be the stabilizer of $\langle AB\rangle$, i.e., $M=\{x\in L| x\langle AB\rangle\subset \langle AB\rangle\}$. Then 
\begin{align*}
\dim_K\langle AB\rangle\geq \dim_K A+\dim_K B-\dim_K M.
\end{align*}
\end{theorem}
\begin{remark}\label{rt''}
It is proved in \cite{20} that the separability condition in Theorem \ref{tt'} can be removed in the price of imposing the assumption $\dim_{K}A\leq 5$.
\end{remark}
The separability condition in Theorem \ref{tt'} is removed in the following Theorem from \cite{8} in the sense that two scenarios turn out for the dimension of $\langle AB\rangle$, depending on the stabilizer of one of the subspaces. That is, in contrast to Theorem \ref{tt'} the subfield $M$ seems to depend on both $A$ and $B$, in Theorem \ref{tt} the stabilizer basically may be seen to depend on only one of the subspaces.

\begin{theorem}\label{tt}
Let $K\subset  L$ be a field extension and $A\subset L$ be a $K$-subspace of $L$ of finite positive dimension. Then 
\begin{itemize}
\item
either for every finite-dimensional subspace $B$ of $L$, we have 
\begin{align*}
\dim_K\langle AB\rangle\geq\dim_K A+\dim_K B-1,
\end{align*}
\item
or there exists an intermediate field $K\subset M\subset L$, such that for every finite-dimensional subspace $B$ of $L$ satisfying 
\begin{align*}
\dim_K\langle AB\rangle <  \dim_K A+\dim_K B -1,  
\end{align*}
we have $\langle  AB\rangle M=\langle AB\rangle$.
\end{itemize}
\end{theorem}

The following theorem from \cite{7} is not as powerful as the linear analogue of Kneser's Theorem but it is useful in classifying unmatchable subspaces with respect to the structural features of subspaces addressed in Theorem \ref{t1.9}.

\begin{theorem}\label{tw}
Let $K\subset L$ be a field extension.  Let $S$ be a $K$-subspace of $L$ containing $1$.
Let $A$ be the 1-atom of $S$. If there exists a $K$-subspace $T$ of $L$ for which 
\[\dim_{K}\langle ST\rangle<\dim_{K} S+\dim_{K}T-1<\dim_{K} L,\]
then $A$ is a subfield of $L$ properly containing $K$.
\end{theorem}

As discussed earlier in Section 1, we addressed a dimension criterion proposed by Eliahou-Lecouvey \cite{15} to characterize matchable bases. This criterion serves as a fundamental cornerstone in our proofs for theorems relating to matchable and unmatchable subspaces. In light of its significance in our work, we present it here for reference.

\begin{theorem}\label{tF}
Let $K\subset L$ be a field extension and $A, B$ be $n$-dimensional $K$-subspaces of $L$. Let $\mathcal{A}=\{a_1,\ldots,a_n\}$ be a basis of $A$. Then $\mathcal{A}$ can be matched to a basis of $B$ if and only if for all $J\subset\{1,\ldots,n\}$, we have
\[\dim_{K}\bigcap_{i\in J}\left(a_i^{-1} A\cap B\right)\leq n-|J|.\]
\end{theorem}
\subsection{Linear matchings in primitive subspaces}
Primitive subspaces are initially introduced in \cite{3} as a linear counterpart to a set of generating elements in a given cyclic group. Given a separable field extension $K\subset L$, a $K$-subspace $A$ of $L$ is called a {\it{primitive subspace}} if $K(a)=L$ for every $a\in A\setminus \{0\}$. Primitive subspaces have been studied in two aspects. First, such subspaces seem interesting in their own right. For example, in \cite{1,4} the size of the largest primitive subspace in a separable field extension is investigated and determined. Second, when dealing with a primitive subspace  $B$, it is proved in \cite{3} that every subspace $A$ with $\dim_{K}A=\dim_{K} B$ is matched to $B$. Theorem \ref{t1.8} part (1) generalizes this result. Moreover, in Section 5, we define Chowla subspaces that are a wider class of vector spaces and speculate on the notion of matchability when replacing primitive subspaces with Chowla subspaces.

\section{Proofs on matching in abelian groups}

\begin{proof}[Proof of Theorem \ref{t1}]
For any subset $S$ of $A$, let $V_S:=\{b\in B\mid Sb\subset A\}$ and $W_S:=V_S\cup\{e\}$. We will first prove that each of conditions (1) through (7) implies $|SW_S|\geq|S|+|W_S|-1$. 

\textbf{Part (1):} Since $B\cap H=\emptyset$, for every proper subgroup $H$ of $G$, we conclude that $W_S\cap H=\{e\}$. Applying Theorem \ref{t*} to $S$ and $W_S$ we have
\begin{align}\label{eq1}
|SW_S|\geq\min\{|G|,|S|+|W_S|-1\}.
\end{align}
On the other hand,
\begin{align*}
&|SW_S|\leq|A|=|B|<|G|,
\end{align*}
since $SW_S\subset A$ and $B\subsetneq G$.

This along with \eqref{eq1} implies that $|SW_S|\geq|S|+|W_S|-1$.

\textbf{Part (2):} Applying Theorem \ref{t**} to $S$ and $W_S$, we have
\begin{align}\label{eq2}
|SW_S|\geq\min\left\{|G|,\left(\frac{l}{l+1}+\frac{1}{2(l+1)}\right)(|S|+|W_S|)\right\}.
\end{align}

In a similar manner as in part (1), we may argue that $|SW_S|<|G|$. Thus
\begin{align}\label{eq3}
|SW_S|\geq \left(\frac{l}{l+1}+\frac{1}{2(l+1)}\right)(|S|+|W_S|).
\end{align}
Clearly $|S|+|W_S|\leq 2n+1$. We claim that $|S|+|W_S|\neq 2n+1$, as otherwise, $|S|=n$ and $|V_S|=n$, implying $S=A$ and $V_S=B$. Thus $B‌=\{b\in B|Ab\subset A\}$ entailing $AB=B$. If follows from Proposition \ref{p***} that $B$ is a subgroup of $G$, contradicting $e\notin B$. So $|S|+|W_S|\leq 2n$. We have:
\begin{align*}
2l+2\geq 2n\geq|S|+|W_S|.
\end{align*}
Therefore, 
\begin{align*}
1\geq&\frac{1}{2l+2}\left(|S|+|W_S|\right)\\
=&\left(1-\frac{l}{l+1}-\frac{1}{2(l+1)}\right)(|S|+|W_S|).
\end{align*}
Hence,
\begin{align}\label{eq4}
\left(1-\frac{l}{l+1}-\frac{1}{2(l+1)}\right)(|S|+|W_S|)\geq|S|+|W_S|-1.
\end{align}
Combining \eqref{eq3} and \eqref{eq4}, we have $|SW_S|\geq|S|+|W_S|-1$, as claimed.

\textbf{Part (3):} Applying Theorem \ref{t**} to $S$ and $W_S$, along with the argument used in part (2), we obtain that
\begin{align}\label{eq5}
|SW_S|\geq \left(1-\frac{l}{l+1}-\frac{1}{2(l+1)}\right)(|S|+|W_S|).
\end{align}
We claim that $|S|+|W_S|\leq 2n-1$. Clearly, $|S|+|W_S|\leq 2n+1$. Additionally, in a similar manner as in part (2),  we may argue that $|S|+|W_S|\neq 2n+1$. So it is left to show $|S|+|W_S|\neq 2n$. Assume to the contrary, $|S|+|W_S|=2n$. This means $|S|+|V_S|=2n-1$. We split into two cases.

Case 1) $|S|=n$ and $|V_S|=n-1$. So $S=A$ and $V_S=V_A$ with $|V_A|=n-1$. Let $B=\{b_1,\ldots,b_n\}$ and without loss of generality, assume $V_S=\{b_1,\ldots,b_{n-1}\}$. Thus, $Ab_i\subset A$, for every $1\leq i\leq n-1$. Since $|Ab_i|=|A|$, we conclude that $Ab_i=A$, for $1\leq i\leq n-1$. Therefore,
\begin{align*}
\prod_{a\in A} a=\prod_{a\in A} ab_1=\prod_{a\in A}ab_2=\cdots=\prod_{a\in A}ab_{n-1}.
\end{align*}
Thus, $a^n=a^nb_i^n$ entailing $b_i^n=e$, $1\leq i\leq n-1$. This contradicts  $B$ having at least two elements with orders greater than $n$.

Case 2) $|S|=n-1$ and $|V_S|=n$. Let $A=\{a_1,\ldots,a_n\}$. Without loss of generality, assume that $S=\{a_1,\ldots,a_{n-1}\}$. Clearly $V_S=B$. Thus, $SB\subset A$. So $a_i B\subset A$, for all $1\leq i\leq n-1$. Since $|a_iB|=|B|=|A|$, hence  $a_iB=A$, for all $1\leq i\leq n-1$. Therefore,
\begin{align*}
\prod_{j=1}^n  a_j=\prod_{j=1}^n a_1b_j=\prod_{j=1}a_2b_j=\cdots=\prod_{j=1}^n a_{n-1}b_j,
\end{align*}
entailing $a_i^n=a_j^n$, for all $1\leq i,j \leq n-1$. This contradicts $|\{a^n|a\in A\}|>2$.

So in both cases, we extract contradictions, and then $|S|+|W_S|<2n-1$. We have 
\begin{align*}
2l+2\geq 2n-2\geq|S|+|W_S|,
\end{align*}
implying
\begin{align*}
1\geq \frac{1}{2l+2}(|S|+|W_S|).
\end{align*}
In a similar manner as in part (2), one may conclude that 
\begin{align}\label{eq6}
\left(\frac{l}{l+1}+\frac{1}{2(l+1)}\right)(|S|+|W_S|)\geq|S|+|W_S|-1.
\end{align}
Combining \eqref{eq5} and \eqref{eq6}, we obtain $|SW_S|\geq|S|+|W_S|-1$, as claimed. 

\textbf{Part (4):} Applying Theorem \ref{ta} to $S$ and $W_S$, we either have $|SW_S|\geq|S|+|W_S|-1$ or $|SW_S|<|S|+|W_S|-1$, and  there exist a subgroup $H\neq\{e\}$ of $G$ with $SW_SH=SW_S$. Suppose that the latter case satisfies. Then we have:
\begin{align}\label{eq7}
|H|\leq|SW_SH|=|SW_S|\leq|A|=p(G).
\end{align}
It follows from \eqref{eq7} that $|H|=p(G)$. Therefore $|SW_SH|=|A|$. On the other hand $SW_SH=SW_S\subset A$. Thus $SW_SH=A$. Choose $x\in SW_S$, then  $xH\subset A$ and since $|xH|=|H|=p(G)=A$, then $xH=A$, contradicting $A$ not being a coset of any nontrivial subgroup of $G$. Therefore $|SW_S|\geq |S|+|W_S|-1$, as claimed. 

\textbf{Part (5):} Since $SW_S\subset A$ and $A$ is a Sidon subset, applying Corollary \ref{cg} to $S$ and $W_S$ one may have $|SW_S|\geq|S|+|W_S|-1$. 

\textbf{Part (6):} We are going to apply Theorem \ref{tx} to $S$ and $V_S$.

We show that  $S\cup SV_S\neq S\langle V_S\rangle$; otherwise since $SV_S\cup S\subset A$, we would conclude that $S\langle V_S\rangle\subset A$. Therefore, if $a\in S$ and $b\in V_S$, $a\langle b\rangle\subset A$. Since $o(b)\geq n$, then $|a\langle b\rangle|\geq|A|$, implying $A=a\langle b\rangle$ and $o(b)=n$. Thus $A=\{a,ab,\ldots,ab^{n-1}\}$, contradicting $A$ not being a progression. Now applying Theorem \ref{tx} to $S$ and $V_S$, we have:
\begin{align*}
|SW_S|=&|S(V_S\cup \{e\})|=|SV_S\cup S|\\
\geq&|S|+\min\{|V_S|,o(b);\  b\in V_S\}=|S|+|V_S|\\
=&|S|+|W_S|-1,
\end{align*}
as desired.

\textbf{Part (7):} 
Let $B_0=B\cup\{e\}$. Then we have:
\[|aH\cap A|+|bH\cap B_0|\leq|H|+1,\]
for every nontrivial finite subgroup $H$ of $G$ and $a,b\in G$. Applying Proposition \ref{p2.3} to $A$ and $B_0$, we conclude that $|SW_S|\geq|S|+|W_S|-1$, as $S\subset A$ and $W_S\subset B_0$.
\\

Therefore in all parts (1)-(7), we established that $|SW_S|\geq|S|+|W_S|-1$, which implies that $|V_S|\leq n-|S|$. On the other hand, $V_S=\underset{a_i\in S}{\bigcap}\left(a_i^{-1}A\cap B\right)$, for every $S\subset A$. Therefore $\left|\underset{a_i\in S}{\bigcap}(a_i^{-1}A\cap B)\right|\leq n-|S|$. Invoking Lemma \ref{l*}, we conclude that $A$ is matched $B$.
\end{proof}

\begin{corollary}\label{c1}
Let $A,B\subset \mathbb{Z}_n$ where $|A|=|B|$ and gcd$(b,n)=1$, for every $b\in B$. Then $A$ is matched to $B$.
\end{corollary}
\begin{proof}
Since gcd$(b,n)=1$ for every $b\in B$, it follows that $B\cap H=\emptyset$, for every proper subgroup $H$ of $\mathbb{Z}_n$ (indeed the condition   gcd$(b,n)=1$ implies that every $b\in B$ is a generating element of $\mathbb{Z}_n$. So if $a\in B\cap H=\emptyset$, we may conclude that $\mathbb{Z}_n=<b>\subset H\subset G$, contradicting $H$ being a proper subgroup of $\mathbb{Z}_n$.) Invoking Theorem \ref{t1} part (1), we conclude that $A$ is matched to $B$.
\end{proof}
\begin{remark}\label{r1}
Theorem \ref{t1} pat (1) is a strengthening of a result from \cite{24}, establishing that all abelian torsion-free groups and cyclic groups of prime orders possess the matching property.
\end{remark}
\begin{remark}
    Theorem \ref{t1} part(2) is a strengthening of Theorem \ref{tf}. Furthermore, a theorem by Losonczy \cite{23} asserts that all abelian torsion-free groups and all cyclic groups of prime order possess the matching property. This can be regarded as a consequence of Theorem \ref{t1} part (2).
\end{remark}

\begin{remark}\label{r2}
Theorem \ref{t1}, part (6) represents a modest strengthening of Theorem \ref{Chowla} wherein the condition ord$(x)\geq|B|+1$, for every $x\in B$ is replaced with ord$(x)\geq|B|$, in the price of imposing the extra assumption of $A$ not being a progression.
\end{remark}
\begin{example}\label{e1}
Consider the subsets $A=\{5,6,7\}$ and $B=\{1,2,3\}$ of $\mathbb{Z}_{15}$. $\mathbb{Z}_{15}$ has two nontrivial proper subgroups:
$H=\{0,5,10\}$ and $K=\{0,3,6,9,12\}$. Upon examination, it becomes evident that neither $A$ nor $B$ can be written as a 2-union of cosets of $H$ (and $K$). By Theorem \ref{t1} part (2), $A$ is matched to $B$.
\end{example}
\begin{example}\label{e2}
Consider the subsets $A=\{1,2,3,9,14\}$ and $B=\{4,5,6,10,13\}$. It can be verified that both $A$ and $B$ satisfy the conditions of Theorem \ref{t1}, part (3), with $l=3$ and $n=5$. Consequently, $A$ is matched to $B$.
\end{example}
\begin{example}\label{e3}
Let $A$ be a maximal Sidon subset of $\mathbb{Z}_{2^m}$. Invoking Proposition 2.8 in \cite{12*}, if $m=7$, then $|A|=12$. This is attained if $A=\{0,1,2,4,8,16,32,64,15,60,101,87\}$. It follows from Theorem \ref{t1}, part (5) that $A$ is matched to every subset of $\mathbb{Z}_{128}$ with $|B|=12$ and $0\notin B$. 
\end{example}
\begin{example}\label{e4}
Consider $A=\{1,8,15,22,30\}\subset \mathbb{Z}_{35}$. Clearly $|A|=5=p(\mathbb{Z}_{35})$ and $A$ is not a coset of the subgroup $K=\{0,7,14,21,28\}$ of $\mathbb{Z}_{35}$. Following Theorem \ref{t1} part (4), $A$ is matched to every subset $B$ of $\mathbb{Z}_{35}$ with $|B|=5$ and $0\notin B$. 
\end{example}

\begin{example}
    Consider the subsets $A=\{0,4,8\}$ and $B=\{3, 6, 8\}$ of $\mathbb{Z}_{9}$. Then for every $a,b \in \mathbb{Z}_{9}$, $|(a+H)\cap A|+|(b+H)\cap B|<4$, where $H=\{0,3,6\}$ is the only proper subgroup of $\mathbb{Z}_{9}$. It follows from Theorem \ref{t1} part (7) that $A$ is matched to $B$.
\end{example}
\subsection{Matchings and coset representatives in abelian groups}
Let $G$ be a finite abelian group and $H$ a subgroup of $G$. If $g_1,…g_m$ are elements such that the cosets $H_{g_1},…H_{g_m}$ are distinct and cover the entirety of $G$, we call these elements {\it{coset representatives.}} Every nonempty subset of a closet representative is called a {\it{semicoset representative}}. Let $A$ and $B$ be two subsets of $\mathbb{Z}_{p^2}$, where $p$ is prime, $|A|=|B|=n\leq p$ and $0\notin B$.Suppose both $A$ and $B$ are coset representatives or semicoset representatives.Then $A$ and $B$ are not contained in the union of $n-1$ cosets of $H$, where $H$ is the subgroup of $\mathbb{Z}_{p^2}$ with $|H|=p$. By Theorem \ref{t1} part (2), $A$ must be matched to $B$.

\begin{proof}[Proof of Theorem \ref{t2}]
Let $A=\{a_1,\ldots,a_n\}$. Since $A$ is not matched to $B$, by Lemma \ref{l*}, there must exist $J\subset \{1,\ldots,n\}$ such that 
\begin{align}\label{eq8}
\left|\bigcap_{i\in J}(a_i^{-1}A\cap B)\right|>n-|J|.
\end{align}
Set $S=\{a_i|i\in J\}$, $V=\underset{i\in J}{\bigcap}(a_i^{-1}A\cap B)$ and $W=V\cup\{e\}$. Then by \eqref{eq8}, one can derive $|V|+|S|>n$. Therefore,
\begin{align}\label{eq9}
|W|+|S|-1>n.
\end{align}
On the other hand, since  $SW\subset A$, it follows from \eqref{eq9} that 
\begin{align}\label{eq10}
|SW|<|W|+|S|-1.
\end{align}
We claim that $\min\{|S|,|W|\}>1$. If this were not the case, we would have either $|S|=1$ or $|W|=1$. If $|S|=1$, it would imply that $S=\{a\}$, for some $a\in A$. Thus,
\begin{align*}
|SW|=|aW|=|W|=|W|+|S|-1.
\end{align*}
This would lead to a contradiction with \eqref{eq10}.

Similarly, if $|W|=1$, it would mean $W=\{e\}$ and we would again have a contradiction, as 
\begin{align*}
|SW|=|Se|=|S|=|S|+|W|-1.
\end{align*}

So $\min\{|S|,|W|\}>1$. Applying Theorem \ref{ty} to $S$ and $W$, we may conclude that either $SW$ forms a progression or $SW$ is quasi-periodic. Since $SW\subset A$, then the first case implies a progression subset for $A$ and the second case yields a quasi-periodic subset for $A$. Both of these cases contradict our initial assumptions on $A$.
\end{proof}

\section{Proofs on matching subspaces in a field extension}

\begin{proof}[Proof of Theorem \ref{t1.8}]
Let $\mathcal{A}=\{a_1,\ldots,a_n\}$ be any bases of $A$. For every nonempty $J\subset\{1,\ldots,n\}$, denote  $S=\langle a_i| i\in J\rangle$, the subspace spanned by $\{a_i|i\in J\}$ and 
\begin{align*}
V_S=\underset{i\in J}{\bigcap}\left(a_i^{-1}A\cap B\right)=\{x\in B\mid  a_ix\in A\ \text{for all}\ i\in J\}.
\end{align*}
Clearly
$\dim_K S=|J|$. Define $W_S=K\oplus V_S$ yielding $\dim_K W_S=\dim_K V_S+1$ and $SW_S\subset A$. We claim that the following inequality holds true in each part of our theorem.
\begin{align*}
\dim_K\langle SW_S\rangle\geq\dim_K S+\dim_K W_S-1,
\end{align*}
and we validate each part independently.

\textbf{Part  1)} Let us begin by assuming that every algebraic element of $L$ is separable over $K$. Applying Theorem \ref{tt'} to $S$ and $W_S$, there must exist an intermediate field $K\subset M\subset L$ such that 
\begin{align}\label{eq11}
\dim_K\langle SW_S\rangle \geq\dim_K S+\dim_K W_S-\dim_K M,
\end{align}
where $M\langle SW_S\rangle=\langle SW_S\rangle$.

Let us introduce $W'_S=\langle M\cup V_S\rangle$. Using Theorem \ref{tt'} once again, we acquire the inequality
\begin{align}\label{eq12}
\dim_K\langle SW'_S\rangle \geq\dim_K S+\dim_K W'_S-\dim_K M',
\end{align}
where $M'$ is the stabilizer of $\langle SW_S\rangle$. 

Now, it becomes evident that $\langle SW'_S\rangle=\langle SW_S\rangle$ because
\begin{align*}
\langle SW'_S\rangle=&\langle S(M\cup V_S)\rangle\\
=&\langle S(M\cup W_S)\rangle\\
=&\langle SM\cup SW_SM'\rangle\\
=&M\langle S\cup SW_S\rangle\\
=&M\langle SW_S\rangle\\
=&\langle SW_S\rangle.
\end{align*}
Given that $\langle SW'_S\rangle=\langle SW_S\rangle$, it is clear that they share the same stabilizer, implying that $M=M'$. This, together with \eqref{eq12}, we obtain
\begin{align*}
\dim_K\langle SW'_S\rangle\geq\dim_K S+\dim_K W'_S-\dim_K M,
\end{align*}
entailing 
\begin{align}\label{eq13}
\dim_K\langle SW_S\rangle\geq& \dim_K S+\dim_K W_S-\dim_K M\notag\\
=& \dim_K S+\dim_K \langle M\cup V_S \rangle -\dim_K M.
\end{align}

Now, utilizing \eqref{eq13} together with the inclusion-exclusion principle for vector spaces, we derive
\begin{align}\label{eq14}
\dim_K\langle SW_S\rangle\geq& \dim_K S+\dim_K M+\dim_K V_S-\dim_K (M\cap V_S)-\dim_K M\notag\\
=&\dim_K S+\dim_K V_S-\dim_K(M\cap V_S).
\end{align}
Since $M\cap B=\{0\}$ and $V_S\subset B$, it follows that $M\cap V_S=\{0\}$. This along with \eqref{eq14}, we conclude that 
\begin{align*}
\dim_K\langle SW_S\rangle\geq \dim_K S+\dim_K V_S=\dim_K S+\dim_K W_S-1,
\end{align*}
as claimed.

A similar line of reasoning may be employed to prove that 
$$
\dim_K\langle SW_S\rangle\geq \dim_K S+\dim_K W_S-1,
$$
in case $n\leq 5$, applicable to general field extensions, in light of Remark \ref{rt''}.

\textbf{Part 2)} Assume, for the sake of contradiction, that
$$
\dim_K\langle SW_S\rangle <\dim_K S+\dim_K W_S-1.
$$
Then by Theorem \ref{tt}, there exists an intermediate field $K\subsetneq M\subset L$ with $M\langle SW_S\rangle=\langle SW_S\rangle$. Given that $\langle SW_S\rangle \subset A$, we may deduce
\begin{align}\label{eq15}
\dim_K M\leq\dim_K M\langle SW_S\rangle =\dim_K\langle SW_S\rangle\leq\dim_K A=p(K,L).
\end{align}
Since $\dim_{K} M\leq p(K,L)$ and $K\subsetneq  M$, it follows that $\dim_{K}M=p(K,L)=n$.  Choose a nonzero element $x\in \langle SW_S\rangle$. Then $Mx\subset M\langle SW_S\rangle =\langle SW_S\rangle \subset A$.  On the other hand, $\dim_K Mx=\dim_K M=n=\dim_K A$. Consequently, we arrive at the contradictory result that $Mx=A$, which contradicts the assumption that $A$ is not a linear translate of any nontrivial intermediate field within  $K\subset M\subsetneq  L$.

\textbf{Part 3) } Suppose, for the sake of contradiction, that
\begin{align}\label{eq16}
\dim_K\langle SW_S\rangle < \dim_K S+\dim_K W_S-1.
\end{align}
By applying Theorem \ref{tt}, we can deduce the existence of an intermediate field $K\subsetneq  M\subset L$ such that $M\langle SW_S\rangle =\langle SW_S\rangle$. Choose a nonzero element $x\in M\setminus K$. Then $x\langle SW_S\rangle\subset \langle SW_S\rangle$. Moreover, considering the dimension, we observe that $\dim_K x\langle SW_S\rangle=\dim_K \langle SW_S\rangle$. Hence $x\langle SW_S\rangle=\langle SW_S\rangle$. Since $A$ is a  Sidon subspace, we have $\dim_K(xA\cap A)\leq 1$. Given that $\langle SW_S\rangle \subset A$, we conclude that
\begin{align*}
1\geq&\dim_K (xA\cap A)\geq \dim_K \left\langle x\langle SW_S\rangle\cap \langle SW_S\rangle \right\rangle\\
=&\dim_K \langle SW_S\rangle =\dim_K \langle S\cup SV_S\rangle \geq\dim_K S.
\end{align*}
As a result, we find that $\dim_{K} S=1$ as $S\neq\{0\}$. Now, it follows from \eqref{eq16} that $\dim_K \langle SW_S\rangle < \dim_K W_S$, which is impossible.

In all parts (1)-(3), we have demonstrated that 
$$
\dim_K\langle SW_S\rangle\geq\dim_KS+\dim_KW_S-1,
$$ 
which implies that $\dim_KV_S\leq n-|J|$, for every $J\subset \{1,\ldots,n\}$. Therefore, according to Theorem \ref{tF}, we conclude that $A$ is matched to $B$.
\end{proof}

\begin{remark}
Theorem 4.2 from \cite{3} is a direct consequence of Theorem \ref{t2}, part (1). In other words, if $B$ is a primitive subspace in a separable field extension $K\subset L$, then it is evident that $B\cap M=\{0\}$, for every intermediate field $K\subset M\subsetneq L$. Therefore, the assumption of Theorem \ref{t1}, part (1) holds, and $A$ is matched to $B$. 
\end{remark}

\begin{remark}
Suppose $K\subset L$ is a field extension that is purely transcendental or $[L:K]=p$, for some prime $p$. In both cases, every algebraic element of $L$ is separable over $K$. That is, in purely transcendental situation, $L\setminus K$ has no algebraic element over $K$, and in the prime degree of extension, the extension is simple and thus separable. Therefore, it follows from Theorem \ref{t1.8} part (1) that $K\subset L$ possesses the linear matching property.
\end{remark}

\begin{example}
Consider the field extension $\mathbb{Q}\subset\mathbb{Q}(\sqrt[35]{2})$. Let $B$ be the $Q$-subspace of $\mathbb{Q}(\sqrt[35]{2})$ spanned by $\sqrt[35]{2}$, $\sqrt[35]{4}$ and $\sqrt[35]{8}$. Then, firstly, $\mathbb{Q}\subset\mathbb{Q}(\sqrt[35]{2})$ is separable, and secondly, $B\cap M=\{0\}$, for every intermediate field $\mathbb{Q}\subsetneq M\subset \mathbb{Q}(\sqrt[35]{2})$. Then by Theorem \ref{t1}, part (1), every 3-dimensional $\mathbb{Q}$-subspace $A$ of $\mathbb{Q}(\sqrt[35]{2})$ is matched to $B$. 
\end{example}

\begin{proof}[Proof of Theorem \ref{t1.9}]
Given that $A$ is not matched to $B$, we can invoke Theorem \ref{tF}, which leads to the existence of a basis $\mathcal{A}=\{a_1,\ldots,a_n\}$ for $A$ and a subset $J\subset\{1,\ldots,n\}$ such that 
\[\dim_{K}\bigcap_{i\in J}\left(a_i^{-1}A\cap B\right)>n-|J|.\]
Set $V_J:=\underset{i\in J}{\bigcap}\left(a_i^{-1}\cap B\right)$, $W_J:=K\oplus V_J$ and  $S_J=\langle a_i\mid i\in J\rangle$. Then we have: 
\[\dim_{K}\langle S_JW_J\rangle<\dim_{K} S_J+\dim_{K} W_J-1.\]
Now, let $M$ denote the 1-atom of $W_J$. We can deduce from Theorem \ref{tw} that $M$ is a subfield of $L$ properly containing $K$. That is, $B\oplus K$ contains a subspace $W_J$ whose $1$-atom $M$ is a subfield of $L$ properly containing $K$, as asserted. 
\end{proof}

\section{Two conjectures}
\subsection{Linear analogue of Theorem \ref{t1}, part (7)}
In Theorem \ref{t1} part (7), we basically deal with matchability between sets $A$ and $B$  whose all nonempty subsets $S\subset A$  and $T\subset B$ satisfy the condition $|ST|\geq|S|+|T|-1$. It is evident that this result relies on Proposition \ref{p2.3}. It is tempting to speculate that Proposition \ref{p2.3} can be extended to a linear setting. If such an extension is successfully established, it could be employed to provide a linear analogue for Theorem \ref{t1} part (7) in some obvious manners. Having said this, we formulate the following conjecture:

\begin{conjecture}
Let $K\subset  L$ be a field extension and $A, B\subset L$ be a $K$-subspace of $L$ of finite positive dimensions. Assume further that for every $a, b \in L$ and every nontrivial proper finite intermediate subfield $K\subset M\subsetneq L$, we have
$$
\dim_{K}(aM\cap A) + \dim_{K} (bM\cap B)\leq [M:K]+1.
$$
Then for all $K$-subspaces $S\subset A$ and $T\subset B$, the following holds
$$
\dim_{K}\langle ST\rangle\geq\dim_{K} S+\dim_{K} T-1.
$$
\end{conjecture}

\subsection{Matching in Chowla Subspaces}

In light of Hamidoune's findings \cite{19} regarding matchable sets $A$ and $B$, with $B$ being a Chowla subset, we aim to formulate a linear analogue of Chowla subsets and explore potential linear matchings in this setting. Consider a field extension $K\subset L$ and let $A$ be a $K$-subspace of $L$. We define $A$ as a Chowla subspace if, for every $a\in A\setminus\{0\}$, we satisfy the condition $[K(a):K]\geq \dim_{K} A+1$. Notably, within a finite separable field extension, every primitive subspace is a Chowla subspace. Therefore, it follows, based on Theorem 2.7 in \cite{1}, that the maximum attainable dimension of a Chowla $K$-subspace in $L$ is greater than or equal to \[[L:K]-\max_{K\subset M\subsetneq L}[M:K],\]
where $M$ is a proper intermediate field.

Now, a natural question to ponder is: In separable (or general) field extensions, what is the upper bound on the size of a Chowla subspace? Additionally, it would be quite valuable 
 if one could investigate the following question, serving as a linear counterpart to Theorem \ref{Chowla}, which asserts that in the abelian group context, if $|A| = |B|$ with $B$ being a Chowla subset, then $A$ must be matched to $B$.

\begin{conjecture}
Let $K\subset L$ be a field extension and $A$ and $B$ be two $n$-dimensional $K$-subspaces of $L$. If $B$ is a Chowla subspace, then $A$ is matched to $B$.
\end{conjecture}
\section{acknowledgement}
We are thankful to David Grynkiewicz for discussions on Kneser's addition theorem and its implications, which have been truly valuable.

\end{document}